\documentclass[a4paper,abstracton,11pt]{scrartcl}
\usepackage[utf8]{inputenc}
\usepackage{amsfonts,amsthm,amssymb,amsmath}
\usepackage[inline]{enumitem}
\usepackage{color}
\usepackage{graphicx}
\usepackage{authblk}

\newtheorem{ctr}{}[section]
\newtheorem{theorem}[ctr]{Theorem}

\newtheorem{corollary}[ctr]{Corollary}
\newtheorem{conjecture}[ctr]{Conjecture}

\newcommand{\Sym}{\operatorname{Sym}}
\newcommand{\Aut}{\operatorname{Aut}}

\title{On asymmetric colourings of graphs with bounded degrees and infinite motion}
\author[1]{Florian Lehner\thanks{Florian Lehner was supported by the Austrian Science Fund (FWF), grant J 3850-N32}}
\author[2]{Monika Pil{\'s}niak\thanks{This work was partially supported by Ministry of Science and Higher Education of Poland and OEAD grant no.PL 08/2017.}}
\author[2]{Marcin Stawiski}
\affil[1]{Institute of Discrete Mathematics, Graz University of Technology, \protect\\ Steyrergasse 30, 8010 Graz, Austria}
\affil[2]{AGH University, Department of Discrete Mathematics, \protect\\al. Mickiewicza 30, 30-059 Krakow, Poland}

\begin{document}

\maketitle
\begin{abstract}
A vertex colouring of a graph is called asymmetric if the only automorphism which preserves it is the identity. 
Tucker conjectured that if every automorphism of a connected, locally finite graph moves infinitely many vertices, then there is an asymmetric colouring with $2$ colours.
We make progress on this conjecture in the special case of graphs with bounded maximal degree. More precisely, we prove that if every automorphism of a connected graph with maximal degree $\Delta$ moves infinitely many vertices, then there is an asymmetric colouring using $\mathcal O(\sqrt \Delta \log \Delta)$ colours.
This is the first improvement over the trivial bound of $\mathcal O(\Delta)$.
\end{abstract}

\section{Introduction}

Call a (not necessarily proper) vertex colouring of a graph \emph{asymmetric} if the only automorphism which preserves it is the identity.
This notion was first introduced by Babai \cite{BAB} in 1977, who proved that $2$ colours suffice for an asymmetric colouring of any regular tree. The concept was later reintroduced by Albertson and Collins \cite{AC} under the name \emph{distinguishing} colouring, and since has received a considerable amount of attention.  

In this paper we investigate connections between the maximal degree $\Delta$ of a graph and the number of colours needed for an asymmetric colouring. It is not hard to see that any connected finite graph has an asymmetric colouring with at most $\Delta+1$ colours, and it was shown independently by Collins and Trenk \cite{CT}, and Klav\v zar, Wong, and Zhou \cite{klavzar} that the only graphs which attain this bound are complete graphs, complete bipartite graphs, and the cycle on $5$ vertices. In 2007, Imrich, Klav\v zar and Trofimov \cite{imrich} extended this result to infinite graphs, showing that any connected infinite graph with maximal degree $\Delta$ has an asymmetric colouring with at most $\Delta$ colours. In \cite{LPS}, the authors of the present paper improved this bound to $\Delta -1$ for infinite graphs with maximal degree at least $3$ and this bound is sharp for every $\Delta \geq 3$.

It is worth noting that all known examples achieving this bound have an automorphism moving only few vertices. This motivates the concept of \emph{motion}. We say that a graph has motion $m$, if the minimal number of vertices moved by a non-trivial automorphism is $m$.
Motion turns out to be a powerful tool for bounding the number of colours needed in an asymmetric colouring. It is often used in the form of the so-called Motion Lemma, which states that if $G$ is finite graph with motion $m$ such that $2^{\frac m 2} \geq |\Aut G|$, then $G$ has an asymmetric colouring with 2 colours. Perhaps the first to explicitly state this result were Russell and Sundaram \cite{russell}, but it is worth pointing out that Cameron, Neumann, and Saxl \cite{primitive} implicitly use a generalisation to permutation groups to show that primitive permutation groups admit an asymmetric colouring with 2 colours, unless they are symmetric groups, alternating groups, or one of finitely many exceptions. The following conjecture due to Tucker \cite{tucker} can be seen as a generalisation of the Motion Lemma to infinite graphs.

\begin{conjecture}[Tucker \cite{tucker}]
\label{con:tucker}
Let $G$ be a connected, locally finite graph with infinite motion. Then there is an asymmetric colouring of $G$ with $2$ colours.
\end{conjecture}
 
Despite numerous partial results towards this conjecture (see for example \cite{smith, lehner-random, lehner-intermediate, lehner-edge, watkins}) it is still wide open in general. In particular, it is still not known whether the conjecture holds for graphs with finite maximum degree. Recently, H\"{u}ning et al.~\cite{draft} proved that it is true for graphs with maximum degree at most 3, in \cite{LPS} this was extended to graphs with maximum degree at most 5, and by using similar arguments one can show that there is an asymmetric colouring for connected graphs with infinite motion with at most $\frac{\Delta}{3}+1$ colours. We note that, although this is a substantial improvement over the trivial bound of $\Delta + 1$ it is still linear in $\Delta$. 

In this short note we show that a graph with infinite motion and maximum degree $\Delta$ admits an asymmetric colouring with $1+(\frac 52 + \frac 32 \log_2 \Delta)\lceil \sqrt \Delta \rceil = \mathcal O(\sqrt \Delta \log \Delta)$ colours. To our best knowledge, this is the first such bound which is sublinear in $\Delta$ thus constituting a significant step towards Conjecture \ref{con:tucker} for graphs with finite maximum degree. 

\section{Definitions and auxiliary results}

Throughout this paper $G$ will denote an infinite, connected graph with vertex set $V$, edge set $E$, and maximal degree $\Delta \in \mathbb N$.  
Let $d$ denote the usual geodesic distance between vertices of $G$. For a vertex $v_0 \in V$ denote the \emph{ball} with centre $v_0$ and radius $k$ by $B(v_0,k) = \{v \in V \mid d(v_0,v) \leq k\}$, and the \emph{sphere} with centre $v_0$ and radius $k$ by $S(v_0,k) = \{v \in V \mid d(v_0,v) = k\}$. As usual we write $N(v)$ for the set of neighbours of a vertex $v$.

For a set $\Omega$, denote by $\Sym(\Omega)$ the group of all bijective functions from $\Omega$ to itself. Further denote by $\Sym_n := \Sym(\{1,\dots,n\})$.

Let $\Gamma$ be a group acting on a set $\Omega$. For $S\subseteq \Omega$, denote the \emph{setwise stabiliser} of $S$ by $\Gamma_S =\{\gamma \in \Gamma \mid \forall s \in S \colon \gamma s \in S\}$ and denote the \emph{pointwise stabiliser} of $S$ by $\Gamma_{(S)} =\{\gamma \in \Gamma \mid \forall s \in S \colon \gamma s = s\}$. Observe that $\Gamma_{(\Omega)}$ is a normal subgroup of $\Gamma$, but the same isn't necessarily true for $\Gamma_{(S)}$ for $S\subsetneq \Omega$. In particular, it makes sense to speak about the quotient $\Gamma / \Gamma_{(\Omega)}$, which is the group of all different permutations induced by $\Gamma$ on $\Omega$. The \emph{stabiliser} of a colouring $c$ of $\Omega$ is defined by $\Gamma_c = \{ \gamma \in \Gamma \mid \forall v \in \Omega : c(v) = c(\gamma v)\}$. A colouring $c$ of $\Omega$ is called \emph{asymmetric} if $\Gamma_c = \Gamma_{(\Omega)}$. Note that in the special case of $\Aut G$ acting on $G$, a vertex colouring is asymmetric if and only if its stabiliser only contains the identity automorphism.

The proof of our main result relies on a corollary to the following result from \cite{subgroupchains}.
\begin{theorem}[Cameron, Solomon, Turull \cite{subgroupchains}]
\label{thm:subgroupchainlength}
The maximal length of a chain of subgroups in $\Sym_n$ is $\lfloor \frac{3n-1}{2}\rfloor- b(n)$, where $b(n)$ denotes the number of ones in the binary representation of $n$.
\end{theorem}

\begin{corollary}
\label{cor:subsetstabiliser}
Let $\Gamma$ be a group acting on two finite sets $\Omega$ and $\Omega'$ respectively. If $\Gamma_{(\Omega')} \subseteq \Gamma_{(\Omega)}$, then there is $S \subseteq \Omega$ such that $|S|\leq \frac{3|\Omega'|}{2}$, and $\Gamma_{(S)} = \Gamma_{(\Omega)}$.
\end{corollary}
\begin{proof}
If $\Gamma$ acts trivially on $\Omega$ we can choose $S = \emptyset$. Otherwise pick $s_1\in\Omega$ with $\Gamma_{(s_1)} \subsetneq \Gamma_{(\Omega)}$. If $\Gamma_{(s_1,s_2,\ldots ,s_i)} \neq \Gamma_{(\Omega)}$, we can inductively pick $s_{i+1} \in \Omega$ which is not stabilised by $\Gamma_{(s_1,s_2,\ldots ,s_i)}$. This process terminates when $\Gamma_{(s_1,\ldots ,s_k)} = \Gamma_{(\Omega)}$.

For $1 \leq i \leq k$, let $\Gamma_i = \Gamma_{(s_1,\dots,s_i)}$. Recall that $\Gamma_{(\Omega')} \unlhd \Gamma$, and since $\Gamma_{(\Omega')}\subseteq \Gamma_{(\Omega)} \subseteq \Gamma_i$, we also have $\Gamma_{(\Omega')} \unlhd \Gamma_i$. In particular we have a chain
\[
    \Sym(\Omega') \supseteq \Gamma / \Gamma_{(\Omega')} \supsetneq \Gamma_1 / \Gamma_{(\Omega')} \supsetneq \dots \supsetneq \Gamma_k / \Gamma_{(\Omega')},
\]
and the Corollary follows from Theorem \ref{thm:subgroupchainlength}.
\end{proof}

It is worth noting that Theorem \ref{thm:subgroupchainlength} (and thus also Corollary \ref{cor:subsetstabiliser}) relies on the classification of finite simple groups. We wish to point out that a slightly weaker bound can be achieved by elementary arguments. Recall that the order of any subgroup divides the order of the ambient group. Since $|S_n| = n! \leq n^n$, any chain of subgroups of $S_n$ has length at most $n \log_2 n$. Using this estimate in the proof of Corollary \ref{cor:subsetstabiliser} leads to an analogous conclusion with $|S| \leq |\Omega'| \log_2 |\Omega'|$.

\section{Proof of the main result}

\begin{theorem}% {\rm({\bf Main Theorem})}
\label{thm:main}
There is a function $\alpha=\mathcal O(\sqrt \Delta \log \Delta)$ such that every connected graph with infinite motion and finite maximal degree $\Delta$ has an asymmetric colouring with at most $\alpha$ colours.
\end{theorem}

\begin{proof}

We will inductively define colourings $c_k$ using colours in the set $\{0,1,2,3,\ldots\} \cup \{\overline 1, \overline 2, \overline 3, \ldots\} \cup \{\infty\}$. Note that this set is infinite, however, we will later determine bounds on the number of colours that we actually use.

Choose an arbitrary root $v_0\in V$. The following properties will be satisfied for every $k\geq 0$.
\begin{enumerate}[label=(\alph*)]
    \item \label{itm:ck-rootcolour} $c_k(v) = 0$ if and only if $v=v_0$,
    \item \label{itm:ck-awaycolour} $c_k(v)=\infty$ if and only if $d(v,v_0) > k$,
    \item \label{itm:ck-restriction} $ c_k|_{B(v_0,k-1)}= c_{k-1}|_{B(v_0,k-1)}$ if $k > 0$,
    \item \label{itm:ck-stabiliser} Let $\Gamma_k$ be the stabiliser of $c_k$ in $\Gamma = \Aut G$, then
    \begin{itemize}
        \item $ \Gamma_k \leq \Gamma_{(B(v_0,k-1))} $ if $k > 0$,
        \item all orbits under $\Gamma_k$ in $B(v_0,k)$ have size at most $\lceil \sqrt{\Delta} \rceil$.
    \end{itemize}
\end{enumerate}
Note that \ref{itm:ck-rootcolour} implies that any automorphism which preserves $c_k$ must fix $B(v_0,k)$ setwise, and thus it makes sense to speak of orbits of $\Gamma_k$ in $B(v_0,k)$.

\subsection{Constructing an asymmetric colouring}

Before we turn to the construction of the colourings $c_k$, let us show that they yield an asymmetric colouring of $G$. By \ref{itm:ck-restriction} we can define a colouring $c_\infty = \lim_{k \to \infty} c_k$ by $c_\infty|_{B(v_0,k)} =c_k|_{B(v_0,k)}$ for every $k$. By \ref{itm:ck-rootcolour}, the only vertex with $c_\infty(v) = 0$ is $v_0$, hence $v_0$ is fixed by any automorphism preserving $c_\infty$. Consequently $B(v_0,k)$ is fixed setwise by any colour preserving automorphism. Since $c_\infty|_{B(v_0,k)} = c_{k+1}|_{B(v_0,k)}$, property \ref{itm:ck-stabiliser} implies that every automorphism in the stabiliser of $c_\infty$ fixes $B(v_0,k)$ pointwise. This is true for every $k$ whence $c_\infty$ is asymmetric.

\subsection{Inductive construction of the colourings}
For $k=0$ let $c_0(v_0) = 0$ and $c_0(v) = \infty$ for $v \neq v_0$. Properties \ref{itm:ck-rootcolour} to \ref{itm:ck-stabiliser} are trivially satisfied for this colouring.

For the inductive definition of $c_{k+1}$, assume that we already have defined $c_k$ with the desired properties. By \ref{itm:ck-awaycolour} and \ref{itm:ck-restriction}, the only vertices where $c_{k+1}$ and $c_k$ can differ are those in $S(v_0,k+1)$, hence it suffices to describe the colouring $c_{k+1}$ on $S(v_0,k+1)$.

Recall that by \ref{itm:ck-rootcolour}, the stabiliser $\Gamma_k$ of $c_k$ must fix $S(v_0,k)$ and $S(v_0,k+1)$ setwise. Let $A_1,\dots,A_n$ be the orbits of $\Gamma_k$ on $S(v_0,k+1)$. Define equivalence relations $R_i$ for $0 \leq i \leq n$ on $S(v_0,k+1)$ by
\[
    x R_i y \iff 
    %\exists \gamma \in \Gamma_k \colon \gamma x = y \text{ and }
    \forall j \leq i\colon N(x)\cap A_j = N(y) \cap A_j.
\]
%For $i=0$, define

%\[
%    x R_0 y \iff 
    %\exists \gamma \in \Gamma_k \colon \gamma x = y \text{ and }
%   N(x) = N(y).
%\]
Note that the condition is void for $i=0$, hence any two vertices are equivalent under $R_0$. Further note that $R_i$ has at most one equivalence class of size larger than $\Delta$, namely the class which has no neighbours in $A_j$ for any $j \leq i$. All other equivalence classes are contained in the neighbourhood of some vertex and thus have size at most $\Delta$. This argument also shows that all equivalence classes with respect to $R_n$ have size at most $\Delta$ since every vertex in $S(v_0,k+1)$ has a neighbour in $S(v_0,k) = \bigcup _{j \leq n} A_j$.

Denote by $\mathcal B_i$ the set of equivalence classes with respect to the relation $R_i$. Observe that $\mathcal B_i$ and $\mathcal B_{i+1}$ are nested, that is, for every $B \in \mathcal B_{i+1}$ there is $B'\in \mathcal B_i$ with $B \subseteq B'$.

Let $x R_i y$ and $\phi \in \Gamma_k$. Then $\phi x R_i \phi y$ because $\phi$ maps neighbours of $x$ and $y$ to neighbours of $\phi x$ and $\phi y$ while fixing every $A_j$ setwise. This shows that $\Gamma_k$ in its natural action on subsets of $S(v_0,k)$ acts on $\mathcal B_i$ by permutations. Denote by $\Gamma_{k,i}$ the pointwise stabiliser of $\mathcal B_i$ in $\Gamma_k$ with respect to this action.

Before we turn to the construction of $c_{k+1}$, we need one last definition. If $c\colon S(v_0,k+1) \to \mathbb N$ is a colouring, then we define the \emph{induced colouring} $c[\mathcal B_i]\colon \mathcal B_i \to \mathbb N$ by
\[
    c[\mathcal B_i](B) = \min_{v\in B} c(v).
\]
Note that the stabiliser of any colouring $c$ satisfies $\Gamma_c \subseteq \bigcap_{1\leq i \leq n} \Gamma_{c[\mathcal B_i]}$.

We now inductively define colourings $c_{k,i}$ of $S(v_0,k+1)$ for $0 \leq i \leq n$ satisfying the following properties.
\begin{enumerate}[label=(\roman*)]
    \item \label{itm:cki-restriction} $c_{k,i}[\mathcal B_j] = c_{k,i-1}[\mathcal B_j]$ for $i\geq 1$ and $j<i$.
    \item \label{itm:cki-monochromatic} Any $B \in \mathcal B_i$ is monochromatic under $c_{k,i}$.
    \item \label{itm:cki-stabiliser} $\tilde \Gamma_{k,i} := \bigcap _{j \leq i} (\Gamma_k)_{c_{k,i}[\mathcal B_j]} \subseteq \Gamma_{k,i}$
\end{enumerate}

Define $c_{k,0} \equiv 1$. The statement \ref{itm:cki-restriction} is void for $i =0$, and properties \ref{itm:cki-monochromatic} and \ref{itm:cki-stabiliser} are clearly satisfied.

Now assume that we already defined $c_{k,i}$ with the desired properties. The only way two vertices can be mapped to each other by $\Gamma_{k,i}$ and are not equivalent in $R_{i+1}$ is, if they have different neighbours in $A_{i+1}$. Hence the stabiliser of $A_{i+1}$ in $\Gamma_{k,i}$ fixes $\mathcal B_{i+1}$ pointwise, and in particular $(\Gamma_{k,i})_{(A_{i+1})} \subseteq \Gamma_{k,i+1}$. Since $\tilde \Gamma_{k,i} \subseteq \Gamma_{k,i}$ we have $(\tilde \Gamma_{k,i})_{(A_{k+1})} \subseteq \tilde \Gamma_{k,i} \cap \Gamma_{k,i+1}  \subseteq (\tilde \Gamma_{k,i})_{(\mathcal B_{i+1})}$, and by Corollary \ref{cor:subsetstabiliser}, there is a subset $\mathcal S \subseteq \mathcal B_{i+1}$ such that 
\[
|\mathcal S| \leq \frac {3|A_{k+1}|}2 \leq \frac {3\lceil \sqrt \Delta\rceil}2 
\]
and 
\[(\tilde \Gamma_{k,i})_{(\mathcal S)} = (\tilde \Gamma_{k,i})_{(\mathcal B_{i+1})}.
\]
Let $\mathcal S = \{S_1,\dots,S_\ell \}$ be a minimal subset (with respect to cardinality) with these two properties. For $j \in \{1,\dots,\ell\}$ and each $v \in S_j$ let $c_{k,i+1}(v) = c_{k,i}(v) + j$. For all other vertices set $c_{k,i+1}(v) = c_{k,i}(v)$. Property \ref{itm:cki-monochromatic} for $c_{k,i+1}$  trivially follows from \ref{itm:cki-monochromatic} for $c_{k,i}$. It remains to show that conditions \ref{itm:cki-restriction} and \ref{itm:cki-stabiliser} hold.

For \ref{itm:cki-restriction}, first note that for any $j \leq i$ and $B'' \in \mathcal B_j$, there is $B' \subseteq B''$ with $B' \in \mathcal B_i$ such that all vertices of $B'$ are coloured with colour $c_{k,i}[\mathcal B_j](B'')$. This follows from \ref{itm:cki-monochromatic} for $c_{k,i}$ and the definition of $c_{k,i}[\mathcal B_j]$. In order to ensure \ref{itm:cki-restriction} for $c_{k,i+1}$, we need to ensure that the colours of some vertices in $B'$ do not change. To this end, it is enough to show that there is $B \subseteq B'$ with $B \in \mathcal B_{i+1} \setminus \mathcal S$.

Assume for a contradiction that every such $B$ is contained in $\mathcal S$. Pick $S \in \mathcal S$ with $S\subseteq B'$ and let $\mathcal S' = \mathcal S \setminus \{S\}$. By minimality of $\mathcal S$ we know that $(\tilde \Gamma_{k,i})_{(\mathcal S)} \subsetneq (\tilde \Gamma_{k,i})_{(\mathcal S')}$. In particular, there must be $\gamma \in (\tilde \Gamma_{k,i})_{(\mathcal S')}$ which acts non-trivially on $S$. Since \ref{itm:cki-stabiliser} holds for $\tilde \Gamma _{k,i}$ we have that $\gamma S \subseteq B'$, but this contradicts the fact that every other $B\in\mathcal B_{i+1}$ with $B \subseteq B'$ is contained in $\mathcal S'$ and thus fixed by $(\tilde \Gamma_{k,i})_{(\mathcal S')}$.

The same argument as above also shows that for $S\in \mathcal S$ and $B'\in \mathcal B_i$ with $S \subseteq B$ we have that $S$ is not stabilised by $\tilde  \Gamma_{k,i}$. In particular, $S$ can be mapped to a different subclass of $B'$ and hence $|S| \leq \frac{|B'|}2$. This fact will be used later to bound the number of used colours.

For the proof of \ref{itm:cki-stabiliser}, note that by definition of $\mathcal S$ we have $(\tilde \Gamma_{k,i})_{(\mathcal S)}\subseteq (\Gamma_{k,i})_{(\mathcal B_{i+1})} = \Gamma_{k,i+1}$. Further note that by \ref{itm:cki-restriction} we have $(\Gamma_k)_{c_{k,i}[\mathcal B_j]} = (\Gamma_k)_{c_{k,i+1}[\mathcal B_j]}$ for every $j \leq i$, and consequently
\[
\tilde \Gamma_{k,i+1} = \tilde \Gamma_{k,i} \cap (\Gamma_k)_{c_{k,i+1}[\mathcal B_{i+1}]} = (\tilde \Gamma_{k,i})_{c_{k,i+1}[\mathcal B_{i+1}]}.
\]
So all we need to show is that
\[
(\tilde \Gamma_{k,i})_{c_{k,i+1}[\mathcal B_{i+1}]} \subseteq (\tilde \Gamma_{k,i})_{(\mathcal S)},
\]
that is, every element of $\tilde \Gamma _{k,i}$ which preserves $c_{k,i+1}[\mathcal B_{i+1}]$ must fix $\mathcal S$ pointwise.

Let $S \in \mathcal S$, let $\gamma \in (\tilde \Gamma_{k,i})_{c_{k,i+1}[\mathcal B_{i+1}]}$ and let $B \in \mathcal B_i$ such that $S \subseteq B$. By \ref{itm:cki-stabiliser} we know that $\gamma \in \Gamma_{i,k}$ and thus $\gamma S \subseteq B$. Since $S \in \mathcal S$ we have (by \ref{itm:cki-monochromatic} for $c_{k,i}$ and the definition of $c_{k,i+1}$) that the colour of the vertices of $S$ is different from the vertices in any subclass of $B$, and hence $\gamma S = S$.

Finally, we obtain $c_{k+1}$ from $c_{k,n}$ as follows. Split every equivalence class in $\mathcal B_n$ into at most $\lceil \sqrt \Delta \rceil$ subsets of size at most $\lceil \sqrt \Delta \rceil$. This is possible because every equivalence class with respect to $R_n$ contains at most $\Delta$ vertices. Now for each $B \in \mathcal B_n$, one of the subclasses keeps the same colour as in $c_{k,n}$, while the others are recoloured with distinct colours in $\overline 1, \overline 2, \dots , \overline{\lceil \sqrt \Delta \rceil}$.

Finally, we need to check that $c_{k+1}$ has the desired properties. By construction, it satisfies \ref{itm:ck-rootcolour}, \ref{itm:ck-awaycolour}, and \ref{itm:ck-restriction}. These properties also immediately imply that $\Gamma_{k+1} \subseteq \Gamma_k$, or more precisely that $\Gamma_{k+1} = (\Gamma_k)_{c_{k+1}}$. The colouring $c_{k+1}$ was constructed from $c_{k,n}$ in a way that every $B \in \mathcal B_n$ still contains some vertices with the same colour as in $c_{k,n}$. Hence every element of $\Gamma_k$ which preserves $c_{k+1}$ must also preserve $c_{k,n}$, whence $\Gamma_{k+1} \subseteq \Gamma_{k,n}$ by \ref{itm:cki-stabiliser} for $c_{k,n}$.

In particular, $\Gamma_{k,n}$ setwise fixes all equivalence classes with respect to $R_n$, and by construction of $c_{k+1}$ from $c_{k,n}$ we conclude that there are no orbits of size more than $\lceil \sqrt \Delta \rceil$ in $S(v_0,k+1)$. The second part of \ref{itm:ck-stabiliser} follows by induction.

For the first part of \ref{itm:ck-stabiliser}, assume for a contradiction that there is $\gamma \in \Gamma_{k+1}$ fixing every $B \in \mathcal B_n$ but not fixing $B(v_0,k)$ pointwise. Then for any $v\in S(v_0,k+1)$, the neighbours of $v$ and $\gamma v$ in $B(v_0,k)$ coincide. Hence we can define an automorphism $\gamma'$  of $G$ by
\[
    \gamma' v = \begin{cases}
    v & \text{if }v \in B(v_0,k),\\
    \gamma v  & \text{otherwise}.
    \end{cases}
\]
This contradicts infinite motion, as $\gamma^{-1} \gamma'$ only moves vertices inside $B(v_0,k)$. Thus we have proved that every element of $\Gamma_{k+1}$ fixes $B(v_0,k)$ pointwise.

\subsection{Number of used colours}

To determine the number of colours used in the colouring procedure, first note that at most $\lceil \sqrt \Delta \rceil$ colours from the set $\{\overline 1, \overline 2,\dots \}$ are used. Hence it only remains to determine the number of colours from $\mathbb N$ used throughout the procedure. 

For this purpose, let $v$ be a vertex with $c_\infty(v) \in \mathbb N$. If $v \in S(v_0,k+1)$, then the colour of $v$ is only changed during the construction of $c_{k,i}$, so it suffices to check how large the colour can get during this construction. 

By definition $c_{k,0}(v)=1$. For $0 \leq i \leq n$, let $B_i \in \mathcal B_i$ be the equivalence class of $v$ with respect to $R_i$. If $c_{k,i}(v) \neq c_{k,i+1}(v)$, we know that $B_{i+1} \in \mathcal S$, and that 
\[
c_{k,i+1}(v) \leq c_{k,i}(v) + |\mathcal S| \leq c_{k,i}(v) +  \frac{3\lceil \sqrt \Delta\rceil}2.
\]

Finally we need to determine how many times it can happen that $B_{i+1} \in \mathcal S$. 
Note that the first time this happens, there must be an automorphism moving $B_{i+1}$ to a different element of $\mathcal B_{i+1}$, and in particular $B_{i+1}$ cannot be the unique class whose size is larger than $\Delta$.
Every subsequent time, we have that $B_{i+1}$ and its image are subclasses of $B_{i}$ of equal size. Consequently $|B_{i+1}| \leq \frac {|B_{i}|}2$, and thus $B_{i+1} \in \mathcal S$ for at most $(1+\log_2 \Delta)$ many $i$. 

Summing up, we get that 
\[
c_{k+1}(v) = c_{k,n}(v) \leq 1 + (1+\log_2 \Delta) \frac{3\lceil \sqrt \Delta\rceil}2. 
\]
Together with the additional $\lceil \sqrt \Delta\rceil$ colours from the set $\{\overline 1, \overline 2,\dots\}$ we have thus used at most $1+(\frac 52 + \frac 32 \log_2 \Delta)\lceil \sqrt \Delta \rceil= \mathcal O(\sqrt\Delta \log \Delta)$ colours as claimed.
\end{proof}

\bibliographystyle{abbrv}
\bibliography{sources}

\begin{thebibliography}{10}

\bibitem{AC}
M.~O. {Albertson} and K.~L. {Collins}.
\newblock {Symmetry breaking in graphs.}
\newblock {\em {Electron. J. Comb.}}, 3(1), 1996.

\bibitem{BAB}
L.~{Babai}.
\newblock {Asymmetric trees with two prescribed degrees}.
\newblock {\em {Acta Math. Acad. Sci. Hung.}}, 29, 1977.

\bibitem{primitive}
P.~Cameron, P.~Neumann, and J.~Saxl.
\newblock On groups with no regular orbits on the set of subsets.
\newblock {\em Archiv der Mathematik}, 43:295--296, 01 1984.

\bibitem{subgroupchains}
P.~J. {Cameron}, R.~{Solomon}, and A.~{Turull}.
\newblock {Chains of subgroups in symmetric groups.}
\newblock {\em {J. Algebra}}, 127(2):340--352, 1989.

\bibitem{CT}
K.~L. {Collins} and A.~N. {Trenk}.
\newblock {The distinguishing chromatic number.}
\newblock {\em {Electron. J. Comb.}}, 13(1):research paper r16, 2006.

\bibitem{draft}
S.~H\"{u}ning, W.~Imrich, J.~Kloas, H.~Schreiber, and T.~W. Tucker.
\newblock Distinguishing graphs of maximum valence $3$.
\newblock {\em {Electron. J. Comb.}}, 26(4):research paper P4.36, 2019.

\bibitem{imrich}
W.~{Imrich}, S.~{Klav\v zar}, and V.~{Trofimov}.
\newblock {Distinguishing infinite graphs.}
\newblock {\em {Electron. J. Comb.}}, 14(1):research paper r36, 12, 2007.

\bibitem{smith}
W.~{Imrich}, S.~M. {Smith}, T.~W. {Tucker}, and M.~E. {Watkins}.
\newblock {Infinite motion and 2-distinguishability of graphs and groups.}
\newblock {\em {J. Algebr. Comb.}}, 41(1):109--122, 2015.

\bibitem{klavzar}
S.~{Klav\v zar}, T.-L. {Wong}, and X.~{Zhu}.
\newblock {Distinguishing labellings of group action on vector spaces and
  graphs.}
\newblock {\em {J. Algebra}}, 303(1):626--641, 2007.

\bibitem{lehner-random}
F.~{Lehner}.
\newblock {Random colourings and automorphism breaking in locally finite
  graphs.}
\newblock {\em {Comb. Probab. Comput.}}, 22(6):885--909, 2013.

\bibitem{lehner-intermediate}
F.~{Lehner}.
\newblock {Distinguishing graphs with intermediate growth.}
\newblock {\em {Combinatorica}}, 36(3):333--347, 2016.

\bibitem{lehner-edge}
F.~{Lehner}.
\newblock {Breaking graph symmetries by edge colourings.}
\newblock {\em {J. Comb. Theory, Ser. B}}, 127:205--214, 2017.

\bibitem{LPS}
F.~{Lehner}, M.~{Pil\'sniak}, and M.~{Stawiski}.
\newblock {Distinguishing infinite graphs with bounded degrees.}
\newblock {\em preprint}, arXiv:1810.03932, 2019.

\bibitem{russell}
A.~{Russell} and R.~{Sundaram}.
\newblock {A note on the asymptotics and computational complexity of graph
  distinguishability.}
\newblock {\em {Electron. J. Comb.}}, 5:research paper R23, 1998.

\bibitem{tucker}
T.~W. {Tucker}.
\newblock {Distinguishing maps.}
\newblock {\em {Electron. J. Comb.}}, 18(1):research paper p50, 21, 2011.

\bibitem{watkins}
M.~E. {Watkins} and X.~{Zhou}.
\newblock {Distinguishability of locally finite trees.}
\newblock {\em {Electron. J. Comb.}}, 14(1):research paper r29, 10, 2007.

\end{thebibliography}

\end{document}